\documentclass{rspublic}
\usepackage{amsfonts,amssymb}
\usepackage{epsfig}
\usepackage{url}

\begin{document}

\title[ Constructing a Proof of the Riemann Hypothesis]{Constructing a Proof of the Riemann Hypothesis}
\author[R.C. McPhedran ]{Ross C. McPhedran$^{1,2}$
}

\affiliation{$^1$ Department of Mathematical Sciences, University of Liverpool,  Liverpool L69 7ZL, United Kingdom\\
$^2$  School of Physics, University of Sydney, Sydney NSW 2006, Australia}
\label{firstpage}

\maketitle
\begin{abstract}{Lattice sums, Riemann hypothesis, Dirichlet $L$ functions, distribution functions of zeros}
This paper compares the distribution of zeros of the Riemann zeta function $\zeta(s)$ with those of a symmetric combination of zeta functions, denoted ${\cal T}_+(s)$, known to
have all its zeros located on the critical line $\Re(s)=1/2$. Criteria are described for constructing a suitable quotient function of these, with properties advantageous for establishing an accessible proof
that $\zeta(s)$ must also have all its zeros on the critical line: the celebrated Riemann hypothesis. While the argument put forward is not at the level of rigour required to constitute a full proof of the Riemann hypothesis, it should convince non-specialists that it must hold.
\end{abstract}
\section{Introduction}
The Riemann hypothesis is renowned as a difficult and important problem in mathematics. For over one hundred and fifty years it has challenged mathematicians to provide a proof, and many notable steps
have been taken towards that goal (Titchmarsh \& Heath-Brown, 1987). The hypothesis states that all zeros of the Riemann zeta function $\zeta(s)$ must lie on the critical line $\Re s=\sigma=1/2$, where  $\Im s= t$. The approach taken here is to investigate the proof of the Riemann hypothesis  by studying the properties of a quotient function, which is composed of a numerator known to have all its zeros on the
critical line, and a denominator equal to $\zeta (2s-1/2)$, together with a balancing function, introduced for reasons to be discussed in Section 2. ThIs approach has previously been used in McPhedran, Williamson, Botten \& Nicorovici (2011; hereafter referred to as I)  to show that, if the Riemann hypothesis holds for any one of a class of angular lattice sums denoted ${\cal C}(1,4 m;s)$, then it holds for all of them, where the lowest
member of the class ${\cal C}(0,1;s)=4 \zeta(s) L_{-4}(s)$ is analytically known (Lorenz, 1871; Hardy, 1920). Here  the notation  $L_{-4}(s)$ refers to a  particular Dirichlet $L$  or beta function, (Zucker and Robertson, 1976). More recently, it has been used for the quotient function $\Delta_5(s)=\zeta(s) L_{-4}(s)/\zeta(2 s-1/2)$ to show that the Riemann hypothesis holds for $ L_{-4}(s)$  if and only if it holds
for $\zeta(s)$ (see McPhedran, 2013, hereafter referred to as II).

The numerator of the quotient function to be used here is the symmetric counterpart (denoted ${\cal T}_+(s)$) of an antisymmetric function ${\cal T}_-(s)$ constructed from two zeta functions with arguments
$2s $ and $2 s-1$. The antisymmetric function had been proved in a posthumous paper by Taylor (1945) to have all its zeros on the critical line. In a companion paper (McPhedran and Poulton, 2013; hereafter III) a proof based on the properties of the quotient function ${\cal T}_+(s)/{\cal T}_-(s)$ is used to show that not only ${\cal T}_-(s)$  but also ${\cal T}_+(s)$ have all zeros restricted to the critical line,
that the zeros of these odd and even functions alternate on that line and that all the zeros are simple. The proof used is based on the formal analogy between the theory of analytic functions and that of the logarithmic potential in electrostatics. Numerical evidence is given in III that the two combinations of zeta functions both have the same distribution functions of zeros on  the critical line $\sigma=1/2$, and that this is also the same as that of ${\cal C}(0,1;s)$ and $\zeta (2 s-1/2)$.

In Section 2, a description is given of the properties required of suitable quotient functions, before proofs are given that the particular function chosen, denoted $\Delta_6(s)$, satisfies the requirements. In Section 3, an argument similar to that used in I and II is developed, leading to the conclusion that all the zeros of $\zeta(s)$ must lie on the critical line. The reasoning set forward in Sections 2 and 3 will use the conventional theorem-proof format, but the emphasis will be on accessibility to a wide class of readers, rather than on a fully rigorous approach. Nevertheless, it is hoped that the arguments set forward here will provide a suitable basis for  a rigorous proof. Sections 2 and 3  include numerical examples of the properties under investigation. These are included to aid the readers' appreciation of the development of ideas, and the parameters have been chosen so that readers can readily construct similar examples using widely available symbolic algebra packages. There exist of course far more powerful special purpose algorithms, such as those which have been used by van de Lune and te Riele (as described in Odlyzko \& te Riele, 1985) to show that the first $1.5\times 10^9$ zeros of $\zeta (s)$ are simple and lie on the critical line.

\section{Requisite Properties of  Quotient Functions and Their Proof for $\Delta_6(s)$}
The arguments relied on in papers  I and II require quotient functions having thee key properties:
\begin{itemize}
\item an appropriate functional equation;
\item source neutrality and
\item a balanced quotient.
\end{itemize}

The quotient should satisfy a functional equation of the type pertaining to $\zeta(s)$, which can be expressed in terms of the even symmetry of the function
\begin{equation}
\xi_1(s)=\frac{\gamma (s/2) \zeta(s)}{\pi^{s/2}}=\xi_1(1-s).
\label{eq2.1}
\end{equation}
The functional equation needs to tightly connect the values of the quotient in $\sigma>1/2$ with those in $\sigma<1/2$.

To establish source neutrality, one needs to show either analytically, or exhibit a priori numerical evidence, that the distribution functions of the zeros of the numerator and denominator functions in the quotient are the same in terms which tend to infinity as $t\rightarrow \infty$. As described above, this property of source neutrality applies for example analytically to the quotient ${\cal T}_+(s)/{\cal T}_-(s)$,
and on the basis of numerical data) to ${\cal T}_+(s)/\zeta(2 s-1/2)$ (see I). The topic of source neutrality occurs in the discussion of periodic Green's functions and the related topic of lattice sums- see for example Poulton, Botten, McPhedran \& Movchan (1999) and Borwein, Borwein \& Shail (1989).

The quotient should also be balanced, using a scaling function if required to ensure that the quotient has the leading term a constant (say unity) as $\sigma\rightarrow \infty$. Ideally, the next term in the asymptotic expansion there should vary in inverse exponential fashion (say as $1/2^{2 s}$).

To use these properties, one establishes that there are infinite number of lines of constant phase of $0$ which leave $\sigma=\infty$, with intervening lines of constant amplitude, and run towards the critical line. As long as the lines of constant phase zero reach the critical line, they cannot enclose between them off-axis zeros of $\zeta(2 s-1/2)$, since there are no off-axis zeros of  ${\cal T}_+(s)$. (One can consider a closed contour composed of  two lines of phase zero, connected by lines of constant $\sigma$: the total change of phase of the quotient function round the contour is zero, and thus in order to enclose poles, it also has to enclose zeros.) Thus, as long as lines of phase zero reach the critical line from $\sigma=\infty$, there will be no off-axis zeros of $\zeta(2s -1/2)$. If one assumes that a line of phase zero fails to reach the critical line, then it and all such lines above it  has to curve up towards $t=\infty$. The functional equation satisfied by the quotient function then guarantees that the image lines
in $\sigma<1/2$ have to have phase and amplitude properties precluded by its asymptotic analysis, a contradiction which guarantees that the lines of phase zero invariably reach the critical line. For further commentary on the three properties prescribed above, see the Appendix.

 The detailed elaboration of this brief summary will now be given. The definition which will be adopted for $\Delta_6(s)$ is
\begin{equation}
\Delta_6(s)=\frac{\xi_1(2 s)+\xi_1(2 s-1)}{\zeta(2 s-1/2) \left( \frac{\Gamma(s)}{\pi^s}+ \frac{\Gamma(s-1/2)}{\pi^{s-1/2}}\right)},
\label{eq2.2}
\end{equation}
or
\begin{equation}
\Delta_6(s)=\frac{\xi_1(2 s)+\xi_1(2 s-1)}{\xi_1(2 s-1/2) {\cal D}(s)},~~{\rm where}~~  {\cal D}(s)=\left[\frac{\pi^{-1/4}\Gamma(s)+\pi^{1/4}\Gamma(s-1/2)}{\Gamma(s-1/4)}\right].
\label{eq2.3}
\end{equation}
Another useful form for asymptotic evaluation of $\Delta_6 (s)$ is
\begin{equation}
\Delta_6(s)=\frac{{\cal A}(s) \xi_1(2 s)+(1-{\cal A}(s))\xi_1(2 s-1)}{\xi_1(2 s-1/2)} ,~~{\rm where}~~  {\cal A}(s)=1/(1+\sqrt{\pi} \Gamma(s-1/2)/\Gamma(s)).
\label{eq2.4}
\end{equation}

\begin{lemma}
The functional equation satisfied by ${\cal A}(s)$ is
\begin{equation}
{\cal A}(1-s)=\frac{1}{1+\tan(\pi s)\left(1/{\cal A}(s+1/2)-1\right)}.
\label{eq2.5}
\end{equation}
The asymptotic expansion for ${\cal A}(s)$ when $|s|$ is large associated with (\ref{eq2.4}) is:
\begin{equation}
{\cal A}(s)\sim1/\left[1+\sqrt{\frac{\pi}{s-1/2}}\left(1+\frac{1}{8(s-1/2)}+\frac{1}{128(s-1/2)^2}-\frac{5}{1024(s-1/2)^3}+O\left(\frac{1}{(s-1/2)^4}\right)\right)\right]
\label{eq2.6}
\end{equation}
\end{lemma}
\begin{proof}
The functional equation (\ref{eq2.5}) follows from the definition (\ref{eq2.4}), together with the functional equation for the $\Gamma$ function. The asymptotic expansion (\ref{eq2.6}) follows from Stirling's formula.
\end{proof}

\begin{theorem}
The analytic function $\Delta_6(s)$
 obeys the functional equation
\begin{equation}
\Delta_6(1-s)={\cal F}_6(s) \Delta_6(s),
\label{eq2.7}
\end{equation} 
where
\begin{equation}
{\cal F}_6(s) =\frac{ {\cal D}(s)}{ {\cal D}(1-s)}.
\label{eq2.8}
\end{equation}
\label{thm6-1}
It is monotonic decreasing along the positive real axis, after its first-order pole at $\sigma=1$. 
On the negative real axis, it has an infinite sequence of first-order zeros at  $\sigma=-2 n,-2n -1/2$ where $n=0,1,2,\ldots$. It has first-order zeros at 1/2 and 3/4. There are also first-order poles, which occur in pairs in intervals $(-2 n,-2n+1/2)$.
Its phase on the critical line, apart from phase jumps occurring at its zeros and poles there,  is well approximated for large $t$ by
\begin{equation}
\arg \Delta_6(\frac{1}{2}+i t)\sim-\frac{\pi}{8}+\frac{\sqrt{\pi/(2 t)}}{\left(1+\sqrt{\pi/(2 t)}\right)} ~({\rm modulo}~\pi).
\label{eq2.9}
\end{equation}
\end{theorem}
\begin{proof}
The functional equation for $\Delta_6(s)$ follows readily from equation (\ref{eq2.3}), which has been written in the form of a numerator which is symmetric under interchange of $s$ and $1-s$, and the first factor in  the denominator which has the same symmetry. We can rewrite the expression (\ref{eq2.8}) using equations (\ref{eq2.3}) and (\ref{eq2.5}):
\begin{eqnarray}
{\cal F}_6(s)&=& \frac{\Gamma(s) \Gamma(3/4-s) {\cal A}(1-s)}{\Gamma(1-s) \Gamma(s-1/4) {\cal A}(s)} \nonumber \\
                     &=& \frac{\Gamma(s) \Gamma(3/4-s)}{\Gamma(1-s) \Gamma(s-1/4) {\cal A}(s)\left[1+\tan(\pi s)\left(\frac{1}{{\cal A}((s+1/2)}-1\right)\right]}.
\label{eq2.10}
\end{eqnarray}

The single pole to the right of $\sigma=1/2$ is that of the  function $\xi_1(2 s-1)$ in the numerator of (\ref{eq2.2}). Poles and zeros to the left of $\sigma=1/2$ then follow from the form (\ref{eq2.10}) for 
${\cal F}_6(s)$: the zeros are at the poles in the numerator of $\Gamma(1-s)$ and of $\tan (\pi s)$, while half the poles are due to the factor $\Gamma (3/4-s)$ in the numerator, the other half being given by the roots of the equation
\begin{equation}
\tan(\pi s)\left(\frac{1}{{\cal A}((s+1/2)}-1\right)=-1.
\label{eq2.11}
\end{equation}

The first-order expansion of $\Delta_6(s)$ near its zero at $s=1/2$ is:
\begin{equation}
\Delta_6(1/2+ \delta s)=\frac{(3\gamma-2\log(2\pi)+\psi(1/2))}{2\zeta(1/2)}\delta s+O[(\delta s)^2].
\label{exteq1}
\end{equation}
Similarly, near the zero at $s=3/4$,
\begin{equation}
\Delta_6(3/4+ \delta s)=\frac{2(\sqrt{\pi}\Gamma(1/4)\zeta(1/2)+\Gamma(3/4)\zeta(3/2))}{\sqrt{\pi}\Gamma(1/4)+\Gamma(3/4)}\delta s+O[(\delta s)^2].
\label{exteq2}
\end{equation}

On the critical line, from (\ref{eq2.7}) and  (\ref{eq2.8}) (modulo $\pi$)  :
\begin{equation}
\arg \Delta_6(1/2+i t)=\frac{1}{2} \arg {\cal F}_6(1/2+i t)=\arg {\cal D}(1/2+i t).
\label{eq2.12}
\end{equation}
Using (\ref{eq2.4}), this gives
\begin{equation}
\arg \Delta_6(1/2+i t)=\arg\Gamma (1/2+i t)-\arg\Gamma(1/4+i t)-\arg{\cal A}(1/2+i t).
\label{eq2.13}
\end{equation}
Using Stirling's formula in (\ref{eq2.13}) yields (\ref{eq2.9}).
 \end{proof}
 
 Graphical examples of the behaviour discussed in Theorem 2.2 can be found in Fig. \ref{fig1}.

\begin{figure}[h]
\includegraphics[width=3.0in]{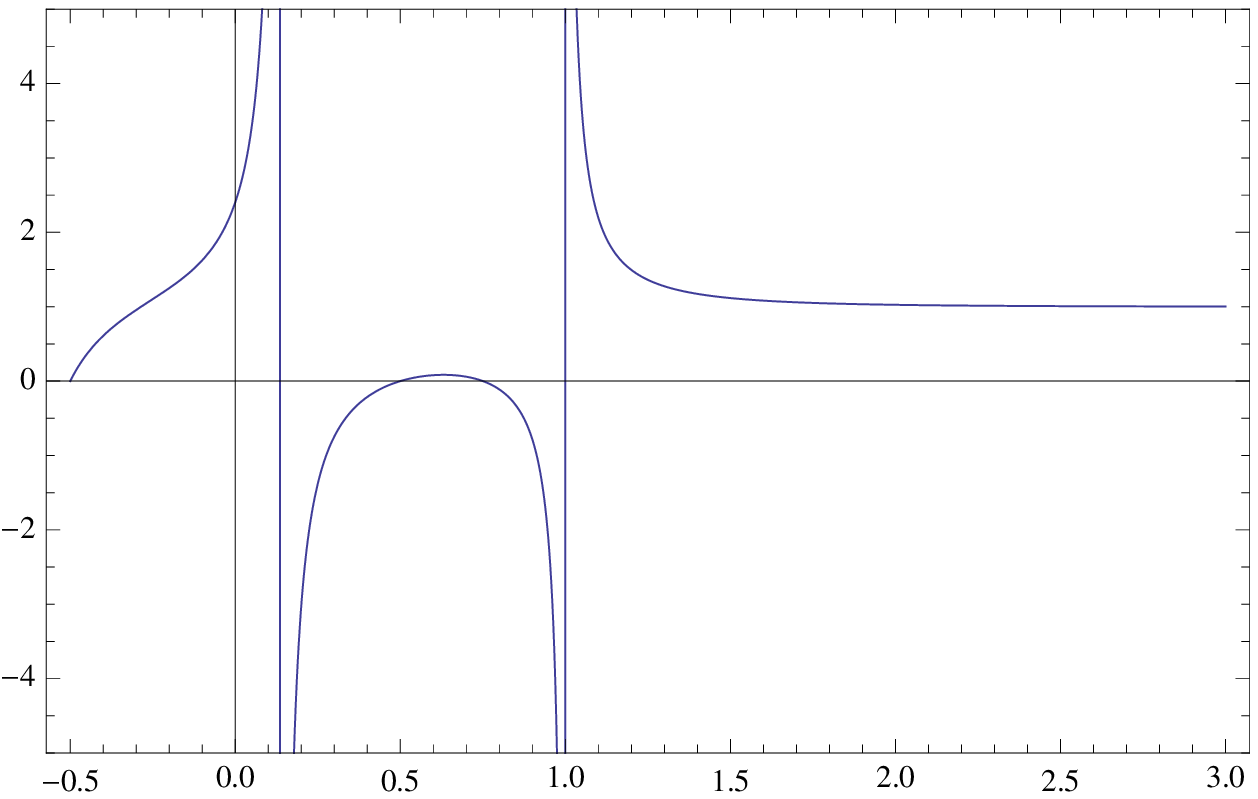}~~\includegraphics[width=3.0in]{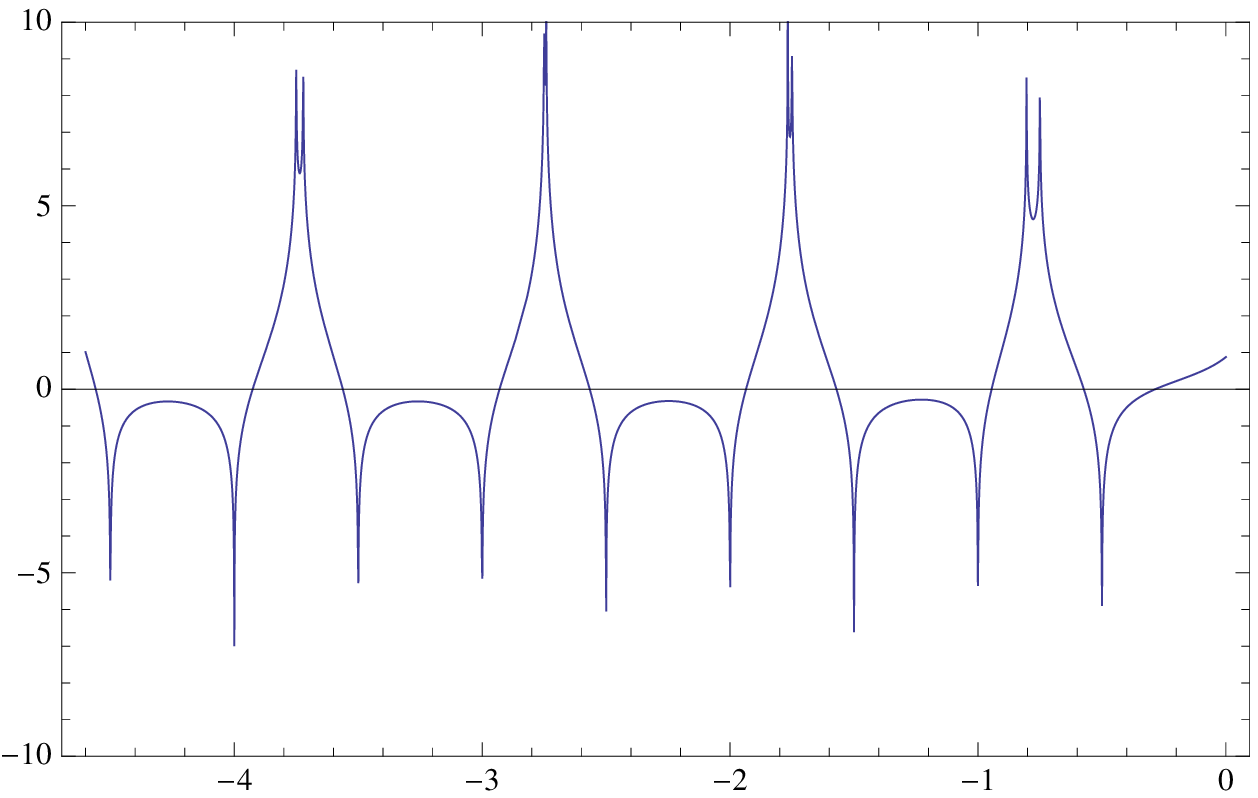}\\
\includegraphics[width=3.0in]{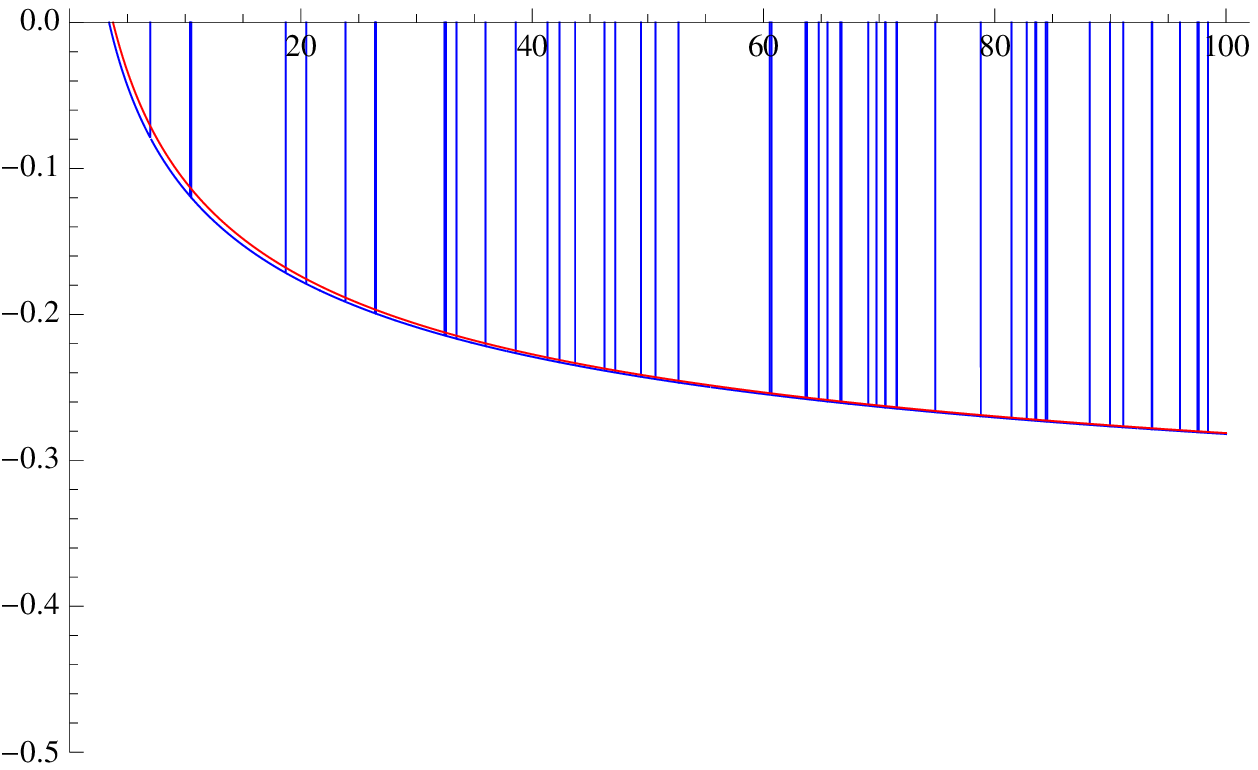}
\caption{Top:Plots of  $\Delta_6(\sigma)$ along the real axis: at left, a linear plot for $\sigma>0$ and at right a logarithmic plot of its modulus for $\sigma<0$. Bottom: a plot of 
$\arg \Delta_6 (1/2+i t)$ (blue line, with phase jumps at zeros and poles; red line, the approximation (\ref{eq2.9}). }
\label{fig1}
 \end{figure}

\begin{figure}[h]
\includegraphics[width=3.0in]{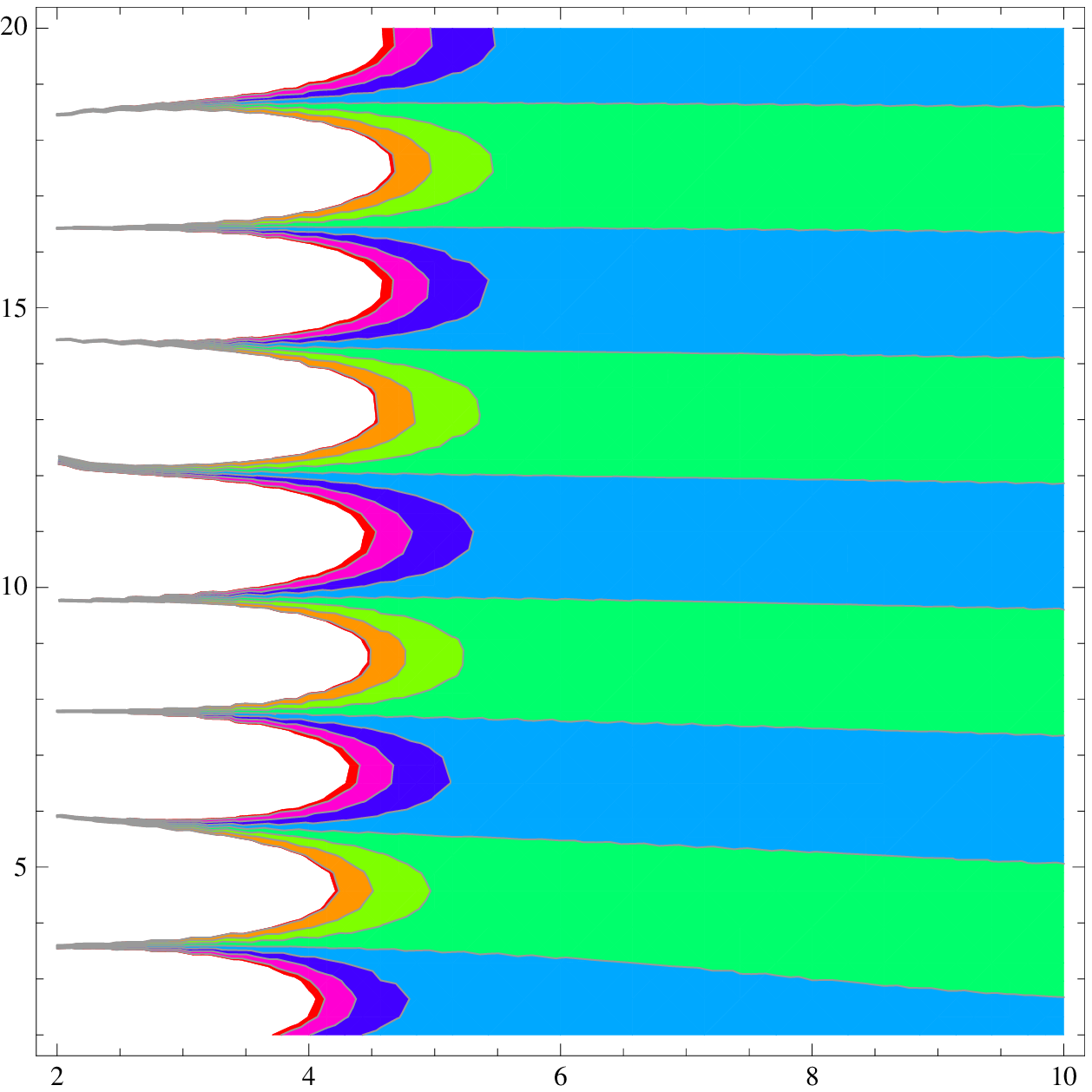}~~\includegraphics[width=3.0in]{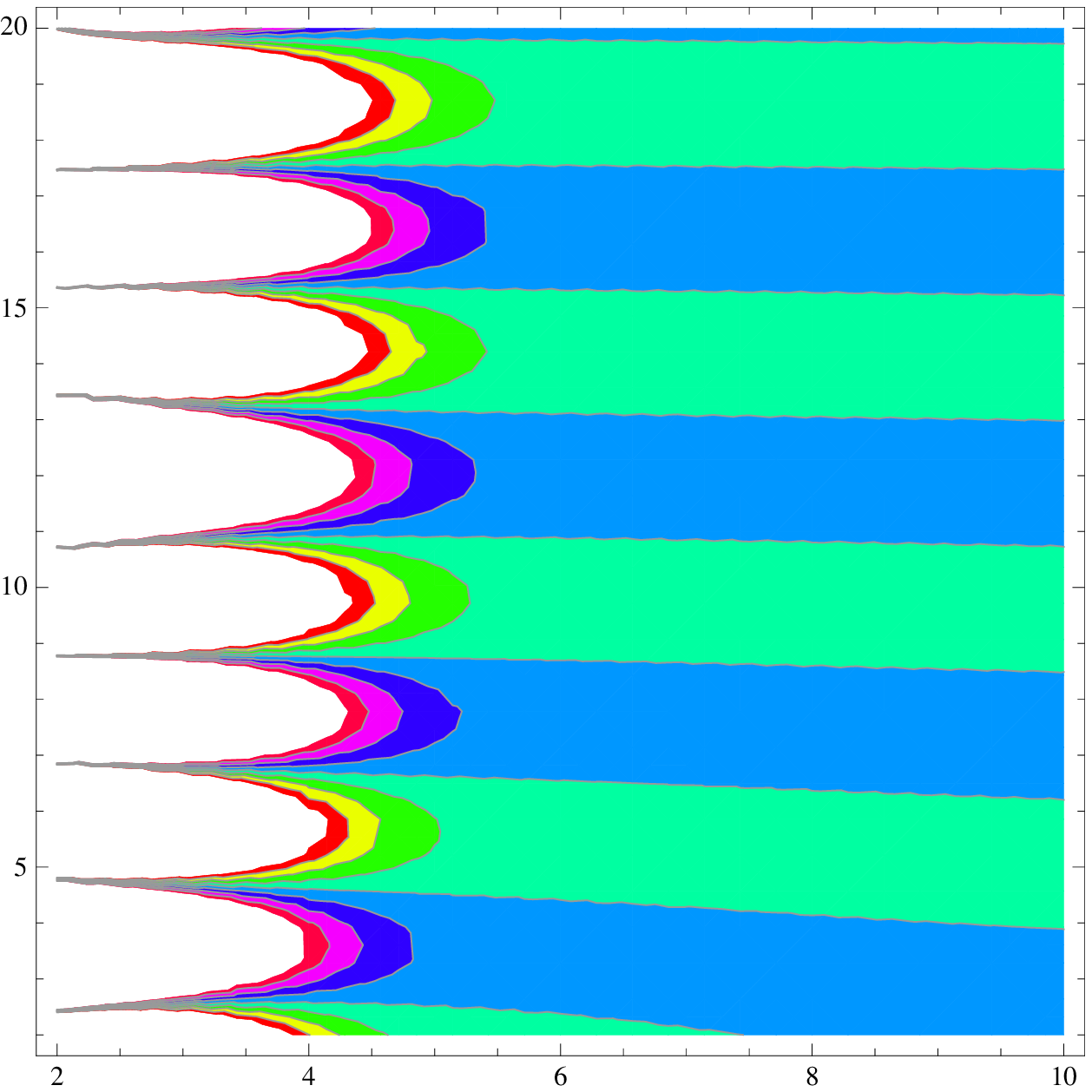}
\caption{ Contour plots of the phase of $\Delta_6(s)$ (left) and its amplitude (right).The regions corresponding to the green colour for $\sigma$ large correspond to negative arguments (left) and amplitudes
less than unity (right), while the blue colour regions indicate positive arguments (left) and amplitudes in excess of unity (right).}
\label{fig2}
 \end{figure}

In Fig. \ref{fig2}, phase and amplitude contour plots are given for $\Delta_6(s)$ in $t>0$ and $\sigma>2$. The phase contours shown correspond to values near zero, with the lines corresponding to phase zero separating light green regions from light blue regions. The similar, but vertically-shifted, contours of constant amplitude at right have amplitude unity at the lines separating light green regions from light blue regions. These contours are the subject of the next theorem. 

\begin{theorem}
The only lines of constant phase of $\Delta_6(s)$ which can attain $\sigma=\infty$ are
equally spaced, and have interspersed lines of constant modulus. All such lines of constant phase
which reach the critical line do so at a pole or zero of  $\Delta_6(s)$.
\label{thm2-2}
\end{theorem}
\begin{proof}
The proposition is true for $t$ not large compared with unity on the basis of numerical evidence
(see Fig. \ref{fig2}).

The leading terms in equation (\ref{eq2.4}) come from the expansion for ${\cal A}(s)$ (see (\ref{eq2.6}))and  the resulting expansion for $1-{\cal A}(s)$:
\begin{equation}
1-{\cal A}(s)=\frac{\sqrt{\frac{\pi}{s-1/2}}\left(1+\frac{1}{8(s-1/2)}+\frac{1}{128(s-1/2)^2}-\frac{5}{1024(s-1/2)^3}+O\left(\frac{1}{(s-1/2)^4}\right)\right)}{\left[1+\sqrt{\frac{\pi}{s-1/2}}\left(1+\frac{1}{8(s-1/2)}+\frac{1}{128(s-1/2)^2}-\frac{5}{1024(s-1/2)^3}+O\left(\frac{1}{(s-1/2)^4}\right)\right)\right]}.
\label{eq2.14}
\end{equation}
These are combined with the series valid in $\sigma>1$ for $\zeta(s)$:
\begin{equation}
\zeta(s)=1+\frac{1}{2^s}+\frac{1}{3^s}+\ldots ,
\label{eq2.15}
\end{equation}
adapted to the three different arguments present in equation (\ref{eq2.4}). The result is
\begin{equation}
\Delta_6(s)=1+\frac{1}{2^{2s}}\left(1-\sqrt{2}+\sqrt{\frac{\pi}{s-1/2}}\right)+\frac{1}{3^{2s}}\left(1-\sqrt{3}+2\sqrt{\frac{\pi}{s-1/2}}\right)+\ldots.
\label{eq2.16}
\end{equation}

From (\ref{eq2.16}) the leading terms in the expansion of 
$\Delta_6(s)$ are
\begin{eqnarray}
\Delta_6(\sigma+i t) &\sim& 1+e^{-\sigma \ln 4}\left\{   \left( 1-\sqrt{2}+\sqrt{\frac{\pi}{2 t}}\right) \cos (t \ln 4)-\sqrt{\frac{\pi}{2 t}} \sin (t\ln 4)\right. \nonumber \\
 & &\left. -i\left[ \sqrt{\frac{\pi}{2 t}} \cos (t \ln 4)+ \sin (t\ln 4)\left(1-\sqrt{2}+\sqrt{\frac{\pi}{2 t}}\right)\right]\right\}
\label{eq2.17}
\end{eqnarray}
Thus, the lines of constant  phase zero of $\Delta_6(\sigma+i t)$ are to leading order, if $t>>\sigma$ and $\sigma>2$, given by
\begin{equation}
\tan (t\ln 4)=\frac{\sqrt{\frac{\pi}{2 t}}}{\sqrt{2}-1-\sqrt{\frac{\pi}{2 t}}}.
\label{eq2.18}
\end{equation}
These occur at values tending  to $n\pi/\ln4$ for integer $n$ as $t$ increases. The lines of amplitude unity of  $\Delta_6(\sigma+i t)$ are to leading order given by
\begin{equation}
\cos (t\ln 4)=0, ~~t=(n+\frac{1}{2})/\ln 4.
\label{eq2.19}
\end{equation}
The lines of constant amplitude unity  lie halfway between the lines of constant phase zero, to leading order.

The lines of constant phase which reach $\sigma=\infty$ are from (\ref{eq2.17}) lines of phase zero. From equation  (\ref{eq2.9}), these can only intersect the critical line at a pole or zero of $\Delta_6(s)$.
The consideration as to whether all such lines must intersect the critical line is left to the next theorem.
\end{proof}

\begin{figure}[h]
\includegraphics[width=3.0in]{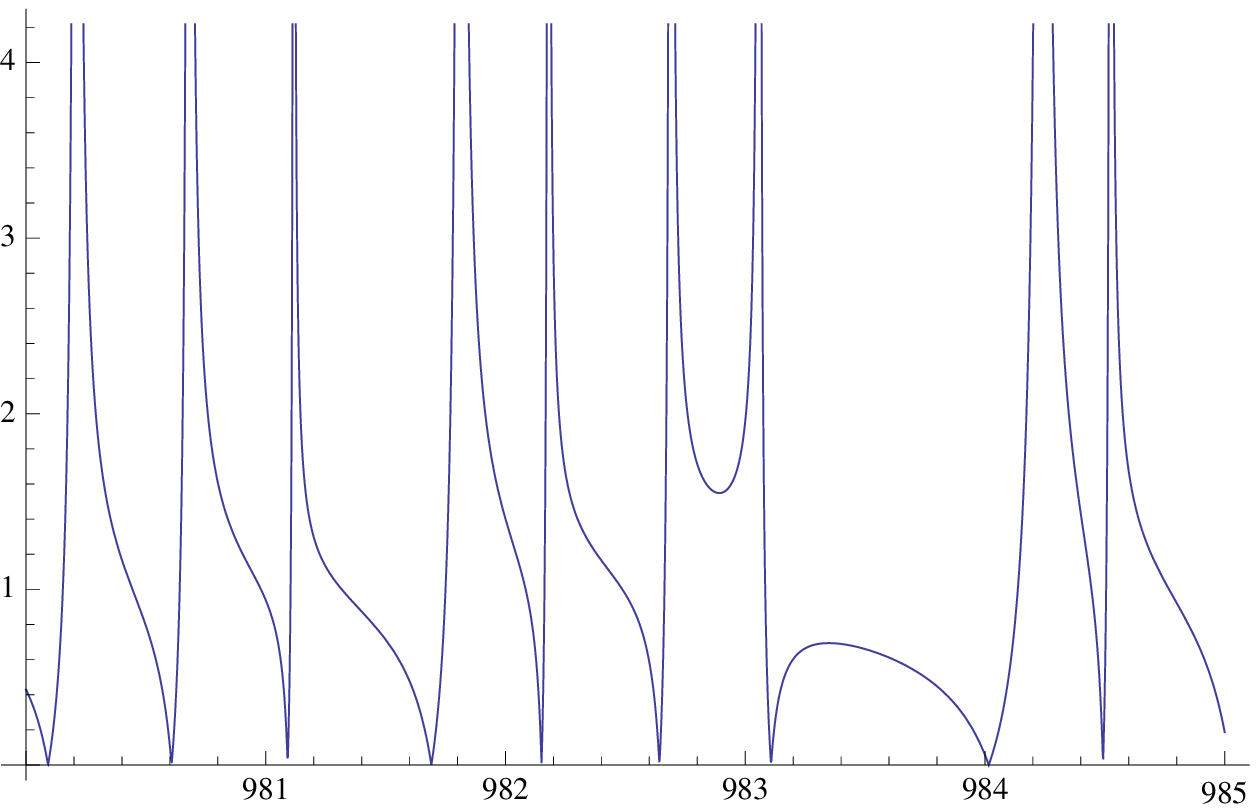}~~\includegraphics[width=3.0in]{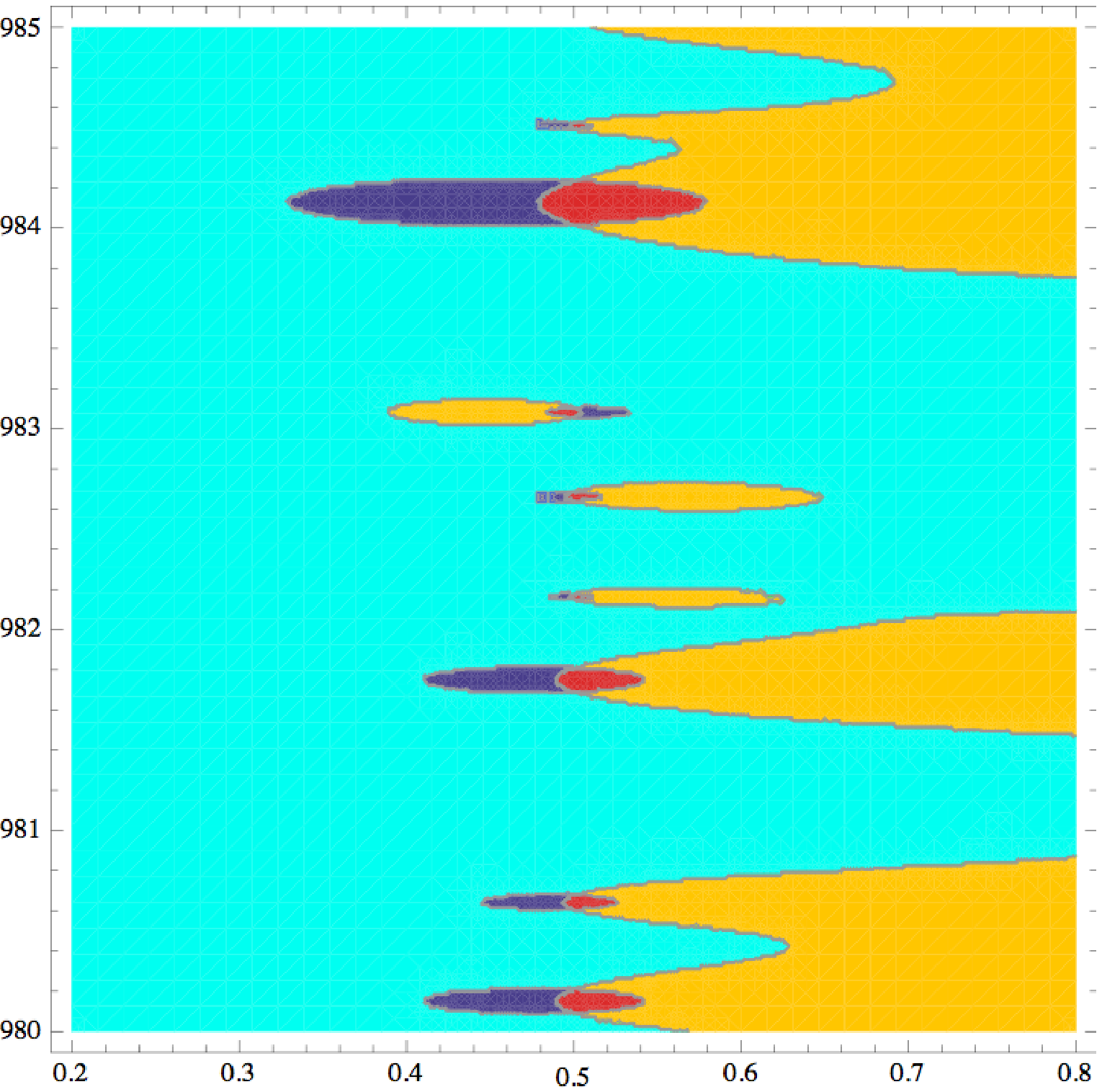}
\caption{ (Left) Plot of the modulus of the amplitude of $\Delta_6(1/2+i t)$ for $t$ between 980 and 985. (Right) A quadrant plot of the phase of $\Delta_6(\sigma +i t)$ for $t$ in the same range and for $\sigma$
lying between 0.2 and 0.8. Phase is indicated by colour, with yellow denoting the first quadrant, red the second, purple the third and light blue the fourth.}
\label{fig3}
 \end{figure}
\begin{figure}[h]
\includegraphics[width=3.0in]{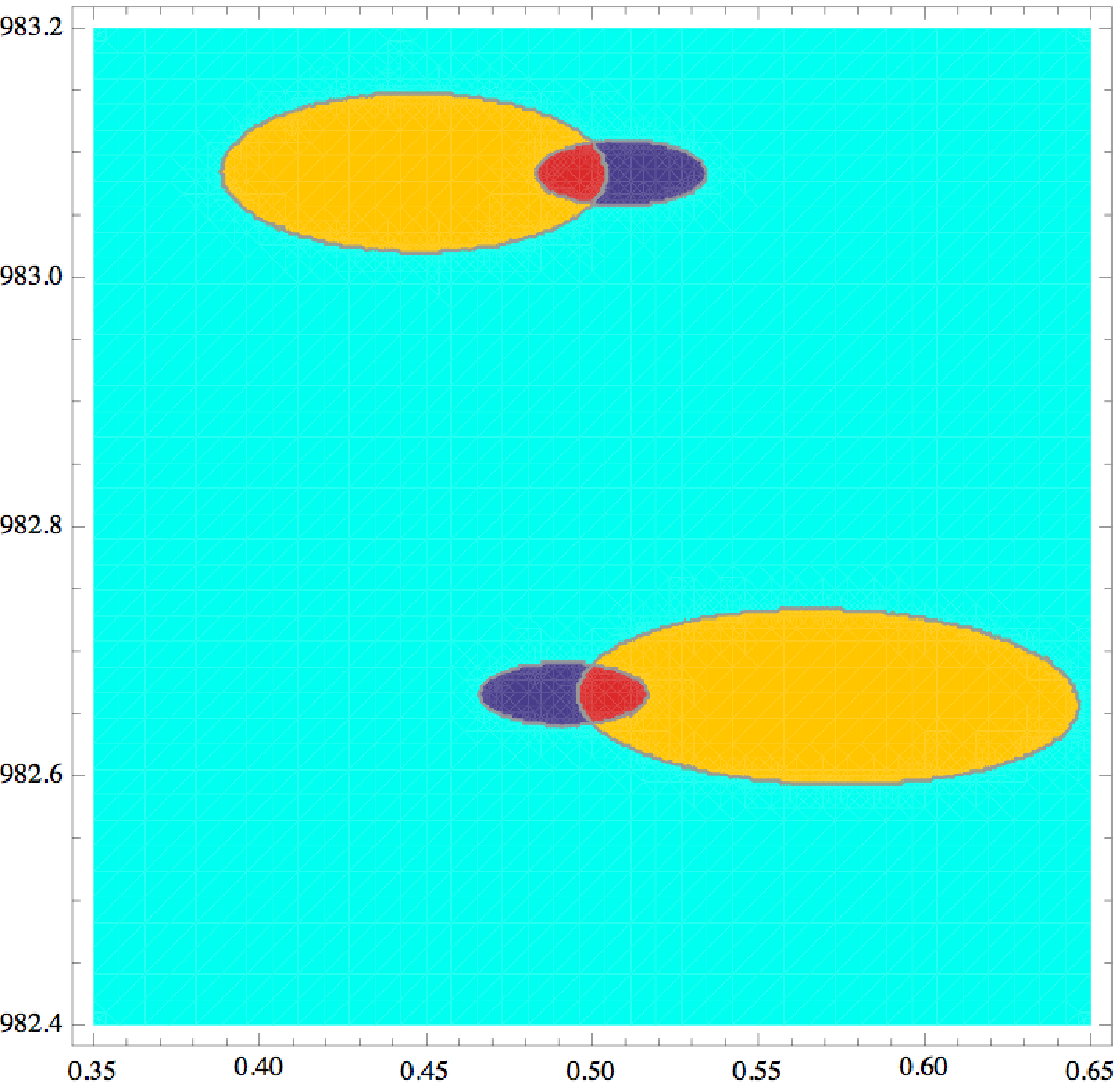}~~\includegraphics[width=3.0in]{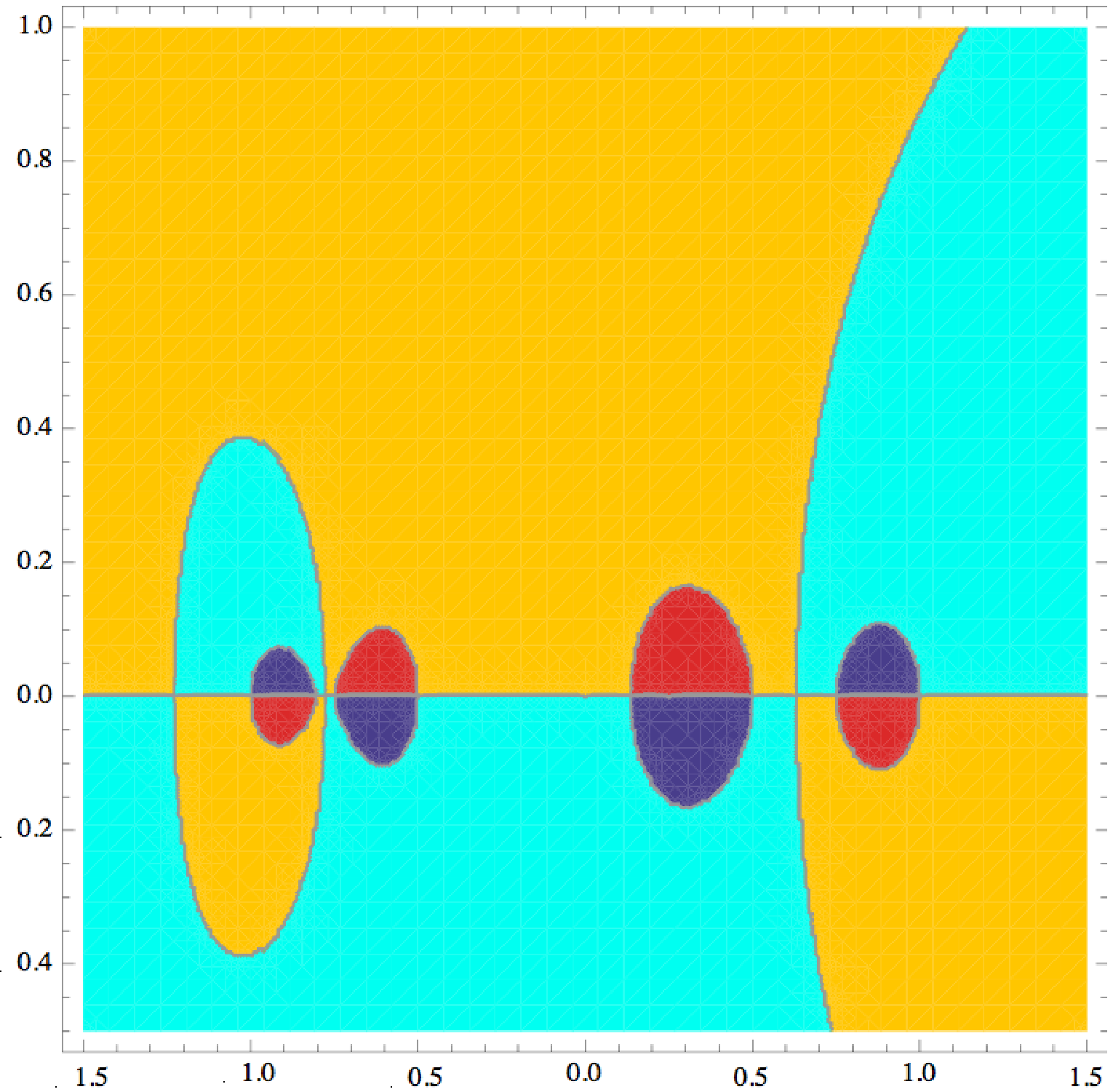}
\caption{ (Left) Detail from Fig. \ref{fig3} showing phase distributions  of $\Delta_6(s)$with the sequence of zeros (Z) and poles (P) being (from bottom to top) ZPPZ. 
(Right) Quadrant phase plots for $\Delta_6(s)$ in the neighbourhood of the real axis.}
\label{fig3ext}
 \end{figure}
Plots of amplitude and phase of  $\Delta_6(s)$ are given in Fig. \ref{fig3}. The range chosen for the amplitude plot along the critical line shows both typical behaviour of zero succeeded by pole, and atypical behaviour of two poles occurring in succession, followed by two zeros in succession. The amplitudes range around unity, unlike the strong exponential behaviour of functions like $\xi_1(1/2+i t)$, as a consequence of the balanced definition (\ref{eq2.2}). The phase plot in Fig. \ref{fig3} shows a single connected region to the left of $\sigma=0.3$, corresponding to the fourth quadrant. The first quadrant regions break up into "islands" in the region where two successive poles occur, followed by two successive zeros. 

More detail of the phase plot of  $\Delta_6(s)$ in the "island" region is given at left in Fig. \ref{fig3ext}. At right, quadrant phase plots are given for the region near the real axis. Running from left to right, the sequence of zeros and poles is: ZPPZPZZP. Two zeros of the derivative of the amplitude function occur splitting the PP pair, and the ZZ pair.

A broader view of the differences in the behaviour of $\Delta_6(s)$ distinctly to the left of the critical line and distinctly to its right is afforded by a comparison of Figs. \ref{fig2} and \ref{fig4}. These show amplitude and phase in two regions connected by a reflection in the critical line. They are connected by an application of the functional equation (\ref{eq2.7}) and a sign reversal of the phase (i.e., a complex conjugation). The strong differences in the morphology of phase and amplitude contours is dictated by the behaviour of the factor ${\cal F}_6 (s)$ in  (\ref{eq2.7}).

\begin{figure}[h]
\includegraphics[width=3.0in]{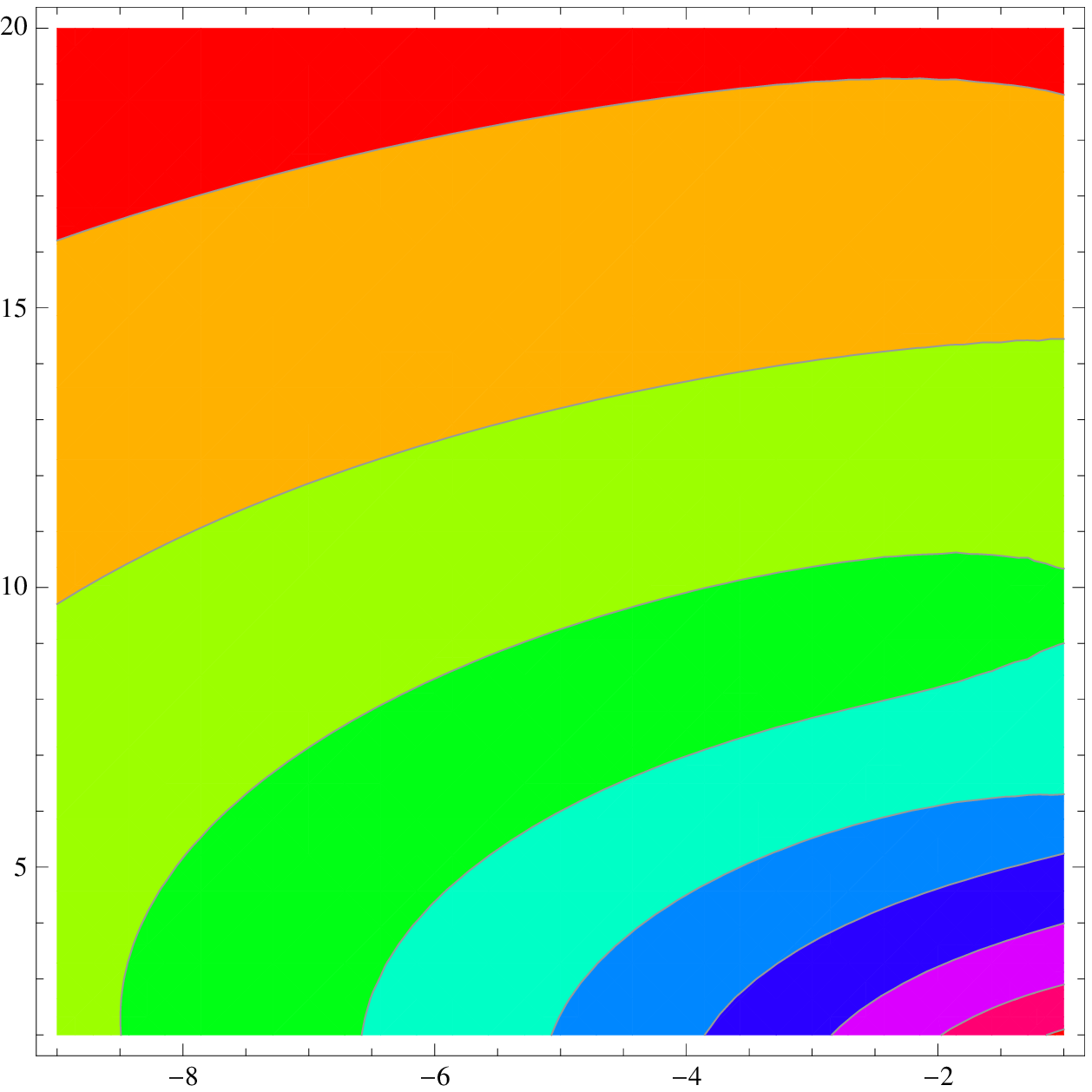}~~\includegraphics[width=3.0in]{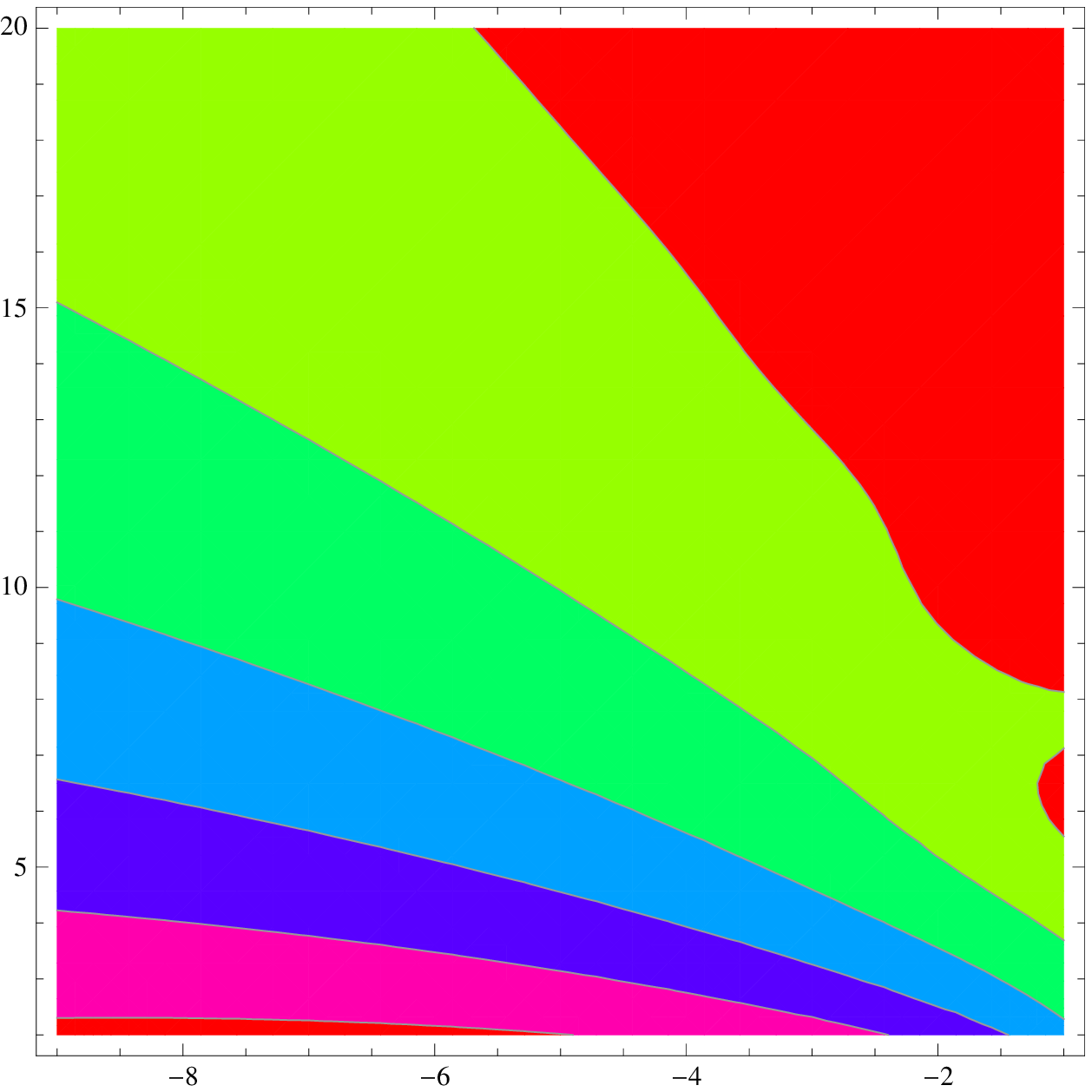}
\caption{ Contour plots of the phase of $\Delta_6(s)$ (left) and its amplitude (right).The phase contours  correspond to values decreasing from 0 (bottom right)  to -0.35 (top left) in steps of 0.05. The amplitude contours correspond to values decreasing from 1.3 (bottom left)  to 1.05 (top right) in steps of 0.05.}
\label{fig4}
 \end{figure}

\begin{theorem} For large $|s|$ the function ${\cal F}_6(s)$ has the asymptotic expansion
\begin{equation}
{\cal F}_6(s)\sim\frac{\sqrt{2}}{1+\cot(\pi s)} \left( 1-\frac{1}{16 s}-\frac{15}{512 s^2}-\frac{75}{8192 s^3}+\ldots\right)\frac{1+{\cal T}_D(s)}{1+\tan (\pi s) {\cal T}_D(s+1/2)}
\label{eq2.20}
\end{equation}
where
\begin{equation}
{\cal T}_D(s)=\sqrt{\frac{\pi}{s-1/2}}\left(1+\frac{1}{8(s-1/2)}+\frac{1}{128(s-1/2)^2}-\frac{5}{1024(s-1/2)^3}+O\left(\frac{1}{(s-1/2)^4}\right)\right).
\label{eq2.21}
\end{equation}
\label{thm2.3}
In consequence, in $\sigma<-2$, $t>2$, the absolute value of $\Delta_6(s)$ exceeds unity, and its phase lies in the fourth quadrant. Also, all lines of  phase zero  of $\Delta_6(s)$ coming from $\sigma=\infty$ intersect the critical line.
\end{theorem}
\begin{proof}
To derive equation (\ref{eq2.20}), we use equations (\ref{eq2.10}) and (\ref{eq2.6}), the latter being ${\cal A}(s)\sim1/(1+{\cal T}_D(s)$, using the notation (\ref{eq2.21}). The ratio $\Gamma(3/4-s)/\Gamma(1-s)$
is rewritten using the functional equation for $\Gamma(s)$. One arrives at the ratio of gamma functions $\Gamma(s)^2/(\Gamma(s-1/4) \Gamma(s+1/4))$, for which the asymptotic expansion occurs in brackets in  (\ref{eq2.20}).

The leading term in the phase of ${\cal F}_6(s)$ comes from $\sqrt{2}/(1+\cot(\pi s)$, which for $t$ large gives $\exp(i\pi/4)$ times terms which are exponentially close to unity. The next contribution
to the phase comes from the ratio involving the ${\cal T}_D$ terms. Taking the leading term for this ratio, we arrive at
\begin{equation}
\arg {\cal F}_6(s)\sim \frac{\pi}{4} -\frac{ \sqrt{2\pi/t}}{1 +\sqrt{\pi/(2 t)}}, ~{\rm for}~\sigma>1/2, t>>1 .
\label{eq2.22}
\end{equation}
This algebraic term dominates the exponentially varying terms in $\Delta_6(s)$ for $\sigma>2$, meaning the phase of $\Delta_6(1-\sigma+ i t)$ varies in a way controlled by $-\arg {\cal F}_6(s)$, and so lies in the fourth quadrant, in keeping with Figs. \ref{fig3} and \ref{fig4}.

The analysis of the variation of the modulus of ${\cal F}_6(s)$ is more delicate than that of its phase. The leading asymptotic terms come from the last factor in (\ref{eq2.20}), which is to good accuracy
$[1+{\cal T}_D(s)]/[1+i{\cal T}_D(s+1/2)]$. The expansion of the modulus of ${\cal F}_6(s)$  is then
\begin{equation}
|{\cal F}_6(s)|\sim \left| 1+\frac{(1-i)\left(\sqrt{\pi/s}-i\pi/s \right)}{1+\pi/s}\right|.
\label{eq2.23}
\end{equation}
Thus, when $t>>\sigma$ and $\sigma$ is positive and not small, $|{\cal F}_6(s)|>1$, and it tends to unity from above as $t$ increases. Once again, the algebraic behaviour of $|{\cal F}_6(s)|$ dominates the exponential behaviour of $\Delta(s)$, so that $|\Delta_6(1-s)|>1$ for $t$ large and positive, and $\sigma$ above and not close to the real axis.

The final statement to be proved is that all lines of phase zero coming from $\sigma=\infty$ reach the critical line.
For large $\sigma$, the lines of phase zero mark the boundary between regions in which the
phase of $\Delta_6(s)$ lies in the fourth quadrant (on one side of the line) and the first quadrant (on the other side).
The amplitude of  $\Delta_6(s)$  in this region is given from (\ref{eq2.16}) by $1+(-1)^n (1-\sqrt{2})/4^{\sigma}$, and thus
increases as $\sigma$ decreases for $n$ odd, and decreases for $n$ even. The amplitude is a monotonic function of $\sigma$, and must continue to behave monotonically even when
$\sigma$ is insufficiently large for (\ref{eq2.16})  to be accurate. Indeed, if say
the amplitude were decreasing, and then started to increase, the lines of constant amplitude forming around the central line of constant phase would have to form closed loops, not possible
unless zeros and poles of $\Delta_6(s)$ lay within the loops. A similar argument of course applies
to the case of increasing amplitude along the line of zero phase.

Turning now to the lines of constant amplitude of unity, on one side the amplitude exceeds unity, and on the other it  is less than unity. The phase of $\Delta_6(s)$ along these lines in the region of large $\sigma$ is given by $-(-1)^n(1-\sqrt{2}) /4^{\sigma}$, and so lies in the fourth quadrant for $n$ odd, and in the first for $n$ even. It increases monotonically in its magnitude as $\sigma$ decreases, and again this conclusion holds true even if $\sigma$ is insufficiently large for (\ref{eq2.16}) to be accurate. 

 For a line of phase zero coming from infinity not to reach the critical line,  lines of constant phase and amplitude would have to curve upwards and run towards $t=\infty$, and they would have to do that
everywhere above the starting value of $t$. (Otherwise, they would cut subsequent regions of lines of constant phase and amplitude heading towards the critical line, requiring the presence of accumulation points of zeros and poles in the finite part of the plane, not in keeping with the analyticity  of $\Delta_6(s)$.) However, such lines coming from $\sigma=\infty$ would require, from the
functional equation (\ref{eq2.7})  corresponding lines to exist mirrored about $\sigma=1/2$. Such lines on the right would have phases in the first and fourth quadrants, and on the left their phase values for $|s|>>1$ would lie in the interval $[-5\pi/8,3\pi/8]$.

The lines of constant amplitude are symmetric under $s\rightarrow 1-s$ if $|s|>>1$. This means sets of lines of constant amplitude are required to start in $\sigma<1/2$ at high $t$ values and proceed as $\sigma$ decreases towards smaller values of $t$. This gives a contradiction, since the lines of constant amplitude curving up
towards the critical line correspond to amplitudes both larger and smaller than unity, while amplitudes in $\sigma<1/2$ and $t$ not small are restricted to amplitudes exceeding unity. In summary, lines of constant amplitude and phase curving up towards infinite $t$ near the critical line are ruled out by the asymptotic behaviours of $\Delta_6(1-\sigma+i t)$ for large $\sigma$ and for large $t$.
\end{proof}

\begin{corollary}
The lines of constant phase coming from $\sigma=\infty$ and reaching the critical line all correspond to strictly the same phase (i.e. their phase is zero, not zero modulo $2\pi$)
\end{corollary}
\begin{proof}
All lines of "phase zero" (meaning possibly zero modulo $2\pi$) of $\Delta_6(s)$ coming from $\sigma=\infty$ reach the critical line. Thus, an infinite number of intervals of $t$ with $\Delta_6(\sigma+i t)$
in the fourth quadrant exist for each value of $\sigma$ down to $\sigma=1/2$, and similarly an infinite number of intervals of $t$ with $\Delta_6(\sigma+i t)$
in the first quadrant exist for each value of $\sigma$ down to $\sigma=1/2$. As $\sigma$ decreases, intervals with the phase lying in the second and third quadrants may also appear, with the strips in the second quadrant then bordered by strips in the first and third quadrants, and strips in the first quadrant bordered by strips in the second and fourth quadrants.  (The occurrence of phase intervals in the second and third quadrant would be required by the failure of the Riemann hypothesis, or by the occurrence of zeros of order two or greater of $\xi_1(s)$.) From equation (\ref{eq2.9}), the critical line between poles and zeros has phase either in the fourth or second quadrants. In an interval of $t$ running from a zero up to a pole, $|\Delta_6(\sigma+i t)|$ is increasing with $t$, so by the Cauchy-Riemann equations, $\partial \arg\Delta_6(\sigma+i t)/\partial \sigma<0$ there. Conversely, in an interval from  a pole up to a zero,  $\partial \arg\Delta_6(\sigma+i t)/\partial \sigma>0$ and lines of constant phase greater than the phase on the critical line in that segment spread to the right of the critical line, while lines of constant but lower phase go left. In this second case, then, the intervals of phase in the fourth quadrant continue smoothly across the critical line, with a similar argument applying in the first case. Thus, the fourth quadrant region bridges the critical line, from its algebraic region to its left, to the exponential region to its right. In the algebraic region, there is no ambiguity of phase, with the phase always lying in the fourth quadrant and being controlled by the phase of ${\cal F}_6(s)$. We can thus link the uniquely specified phase in the fourth quadrant in $\sigma<<1/2$ to that in $\sigma>>1/2$, which proves the result.
\end{proof}

\section{The Riemann Hypothesis for $\zeta(2 s-1/2)$}

Using Theorem \ref{thm2.3},  the argument of I can be followed, to prove the equivalent of its  Theorem 4.3. In the case of $\Delta_6(s)$, it is known from III  that the numerator function ${\cal T}_+(s)$ obeys the Riemann hypothesis.
\begin{theorem}
Given  that the function ${\cal T}_+(s)$ obeys the Riemann hypothesis,  then $\zeta(2 s-1/2)$ obeys the Riemann hypothesis. 
\label{thm3.1}
\end{theorem}
\begin{proof}
Consider lines $L_1$ and $L_2$  in $t>0$ along which the phase of $\Delta_6(s)$  is zero. Join these
lines with two lines  to the right of the critical line along which  $\sigma$ is constant. Then the change of argument of  $\Delta_6(s)$ around the closed contour  $C$ so formed is zero, so by the Argument Principle the number of poles inside the contour equals the number of zeros. Each non-trivial zero is formed by a zero of ${\cal T}_+(s)$, and as  there are no such zeros within $C$ there can be no zeros of  $\zeta(2 s-1/2)$   within $C$.

These arguments prove the theorem in the region to the right of the critical line lying between lines of zero phase of $\Delta_6(s)$, with the result to the left of the critical line then guaranteed by the functional equation (\ref{eq2.7}) .

To complete the proof one must show that any point in the region $\sigma>1/2$, $t>0$ is enclosed between lines of phase zero of  $\Delta_6(s)$ coming from $\sigma=\infty$. This is evidently the case, since several such lines just above $t=0$ are shown in Fig. \ref{fig2}. In addition, for any $\sigma>1/2$ the infinite number of such constant phase lines cannot cluster into a finite interval of $t$, since that would indicate an essential singularity of  $\Delta_6(s)$ for that $\sigma$.
\end{proof}

\begin{theorem}
Given the Riemann hypothesis holds for $\zeta(s)$ and ${\cal T}_+(s)$, then given any two lines of phase zero of $\Delta_6(s)$ running from $\sigma=\infty$ and intersecting the critical line, the number of zeros and poles of $\Delta_6(s)$ counted according to multiplicity and lying properly between the lines must be the same.
\label{thm3-2}
\end{theorem}
\begin{proof}
Consider  a contour composed of two lines of phase zero, the segment between them on the critical line, and a segment between them
on the interval $\sigma>>1$. The total phase change  around this contour is strictly zero, since the region $\sigma>>1$ has the phase of $\Delta_6(\sigma+i t)$ constrained to be close to zero. Let $P_u=(1/2,t_u)$ be the point at the upper end of the segment on the critical line, and $P_l=(1/2,t_l)$ be the point at the lower end.

The total phase change between a point approaching $P_u$ on the contour from the right and a point leaving $P_l$ going right is, by construction, zero. This phase change is made up of contributions from the changes of phase at the zero or pole $P_u$, from the zero or pole $P_l$, from the $N_z$ zeros and $N_p$ poles on the critical line between $P_u$ and $P_l$, and from the phase change between the zeros and poles. In this list, the first phase change is $\phi_6(t_u)$, the phase of $\Delta_{6}(s)$ on the critical line just below $t_u$. The second change is $-\phi_6(t_l)$, where $\phi_6(t_l)$ is the phase of 
$\Delta_6(s)$  just above $t_l$. Giving zero $n$ a multiplicity $z_n$, and pole $n$ a multiplicity $p_n$, the phase change at the former is $-\pi z_n$ and at the latter $\pi p_n$. The phase change between zeros and poles is $\phi_6(t_l)-\phi_6(t_u)$. Hence the phase difference is
\begin{equation}
\phi_6(t_u)-\phi_6(t_l)-\pi \left(\sum_{n=1}^{N_z}z_n-\sum_{n=1}^{N_p}p_n\right)+\phi_6(t_l)-\phi_6(t_u)=0,
\label{eq3.1}
\end{equation}
leading to
\begin{equation}
\sum_{n=1}^{N_z}z_n=\sum_{n=1}^{N_p}p_n ,
\label{eq3.2}
\end{equation}
as asserted.

\end{proof}
\begin{corollary}
Given  the Riemann hypothesis holds for $\zeta(s)$ or $L_{-4}(s)$, if all zeros of $\zeta(s)$ on the critical line have multiplicity one, the numbers of zeros and poles on the critical line between any pair of lines of phase zero of $\Delta_6(s)$ coming from $\sigma=\infty$ are the same. The distribution functions for zeros of ${\cal T}_+(s)$ and $\zeta (2s -1/2)$
on the critical line
must then agree in all terms which   do not go to zero as $t\rightarrow \infty$.
\label{corol5-3}
\end{corollary}
\begin{proof}
The first assertion is a simple consequence of theorem \ref{thm3-2}. The second also follows from
theorems \ref{thm3-2} : the number of zeros and poles between successive zero lines coming from $\sigma=\infty$ match for all such pairs of lines. There are no zeros or poles on the critical line before the first line of phase zero off the real axis coming from $\sigma=\infty$,
so the distribution functions of zeros of the numerator and denominator of $\Delta_6(s)$
must then agree exactly at  the infinite set of values of $t$ corresponding to the intersection points above this first line, and in neighbourhoods including each member of the set.  This precludes
those distribution functions differing by terms which do not go to zero as $t\rightarrow \infty$.
\end{proof}

From Titchmarsh and Heath-Brown (1987), given the Riemann hypothesis holds for $\zeta(s)$,
the distribution function for its zeros on the critical line is
\begin{equation}
N_\zeta(\frac{1}{2},t)=\frac{t}{2\pi} \log (t)-\frac{t}{2\pi}(1+\log(2\pi))+I_\zeta(\frac{1}{2},t)+F_\zeta(\frac{1}{2},t),
\label{eq3.3}
\end{equation}
where the functions $I$ and $F$ denote the sums of all terms which go to infinity as $t\rightarrow \infty$ or which remain finite, respectively. Hence, 
\begin{equation}
N_\zeta(\frac{1}{2},2 t)=\frac{t}{\pi} \log (t)-\frac{t}{\pi}(1+\log(\pi))+I_\zeta(\frac{1}{2},2 t)+F_\zeta(\frac{1}{2},2 t),
\label{eq3.4}
\end{equation}
From the result of corollary \ref{corol5-3},  for the distribution function of the zeros of ${\cal T}_+(s)$ on the critical line, it follows that
\begin{equation}
N_{+}(\frac{1}{2},t)=\frac{t}{\pi} \log (t)-\frac{t}{\pi}(1+\log(\pi))+I_{+}(\frac{1}{2},t)+F_{+}(\frac{1}{2},t),
\label{eq3.5}
\end{equation}
where 
\begin{equation}
I_\zeta(\frac{1}{2},2 t)=I_{+}(\frac{1}{2},t).
\label{del5-15}
\end{equation}
Thus, the distribution function for the zeros of $\zeta(s)$ determines that for ${\cal T}_+(s)$, in all important terms.

The work of R.C. McPhedran on this project has been supported by the Australian Research Council's Discovery Grants Scheme. 
He also acknowledges the financial support of the European Community's Seventh Framework
Programme under contract number PIAPP-GA-284544-PARM-2. M

\section*{Appendix: Constraints on Counterexamples}
Counterexamples provide a useful way of checking proofs, and also of understanding the assumptions on which they rely. In Section 2, three important properties were listed which were satisfied by the definition of
$\Delta_6(s)$. The results given in Sections 2 and 3 will be illustrated briefly here using numerical examples based on algebraic modifications of $\Delta_6(s)$, which might be considered capable of providing counterexamples.

The first modification illustrated consists of inserting zeros off the critical line in the denominator of $\Delta_6(s)$, in such a way as to preserve its functional equation. The zeros of the denominator are artificially inserted using the multiplicative factor:
\begin{equation}
{\cal D}_6(s)=(s-s_0)(s-(1-s_0))(s-\overline{s_0})(s-(1-\overline{s_0})),
\label{app1}
\end{equation}
with $s_0$ specifying the additional zero, chosen to lie off the critical line. The results of this modification are shown in Fig. \ref{figapp1}. The extra zeros in the denominator have destroyed source neutrality, resulting
in the pause regions corresponding to quadrants 1, 2 and 3 becoming extended far to the left of the critical line, in a way not permitted by the asymptotic result inTheorem 2.4. The destruction of source neutrality also results in the region of phase corresponding to the fourth quadrant becoming multiply rather than simply connected.

\begin{figure}[h]
\includegraphics[width=3.0in]{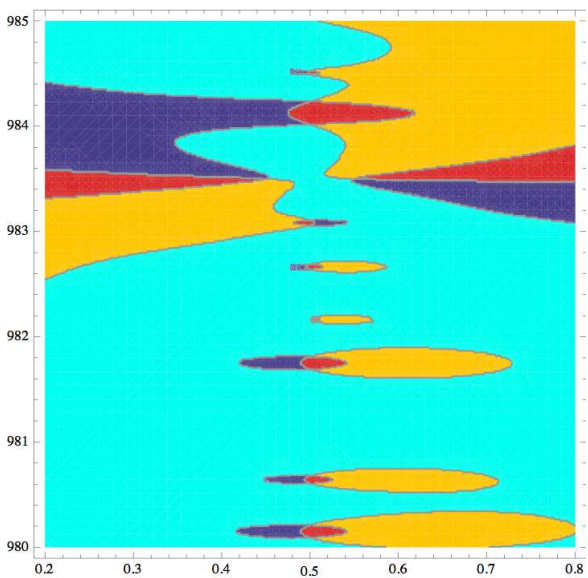}
\caption{ Contour plot of the phase of $\Delta_6(s)$, with its denominator having artificial off-axis zeros introduced as described in equation (\ref{app1}). Here $s_0=0.45+983.5 i$.}
\label{figapp1}
 \end{figure}

The second modification exemplified consists of restoring source neutrality, by multiplying the numerator of $\Delta_6(s)$ by a factor introducing four zeros on the critical line:
\begin{equation}
{\cal N}_6(s)=(s-s_1)(s-s_2))(s-\overline{s_1})(s-\overline{s_2}),
\label{app2}
\end{equation}
with $s_1$ and $s_2$ chosen on the critical line, close to the off-axis zeros of the denominator. As shown in Fig. \ref{figapp2}, this modification gives an isolated region containing the two off-axis poles of  
$\Delta_6(s)$ and the two on-axis zeros. The phase region corresponding to quadrant 4 is  simply connected once more, but first quadrant phase region no longer extends  to $\sigma=\infty$, as required by Theorem 2.3.

\begin{figure}[h]
\includegraphics[width=3.0in]{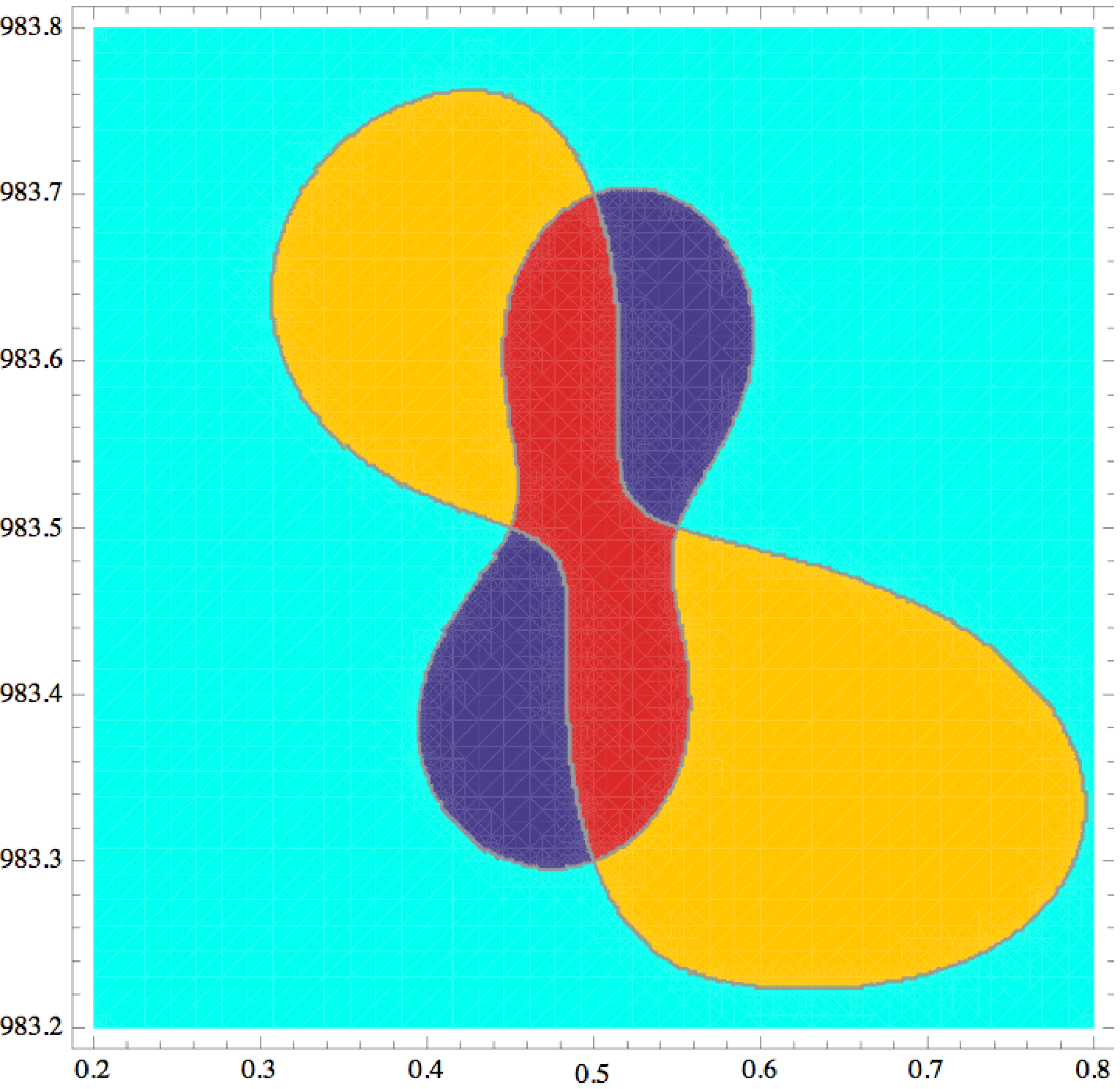}
~~\includegraphics[width=3.0in]{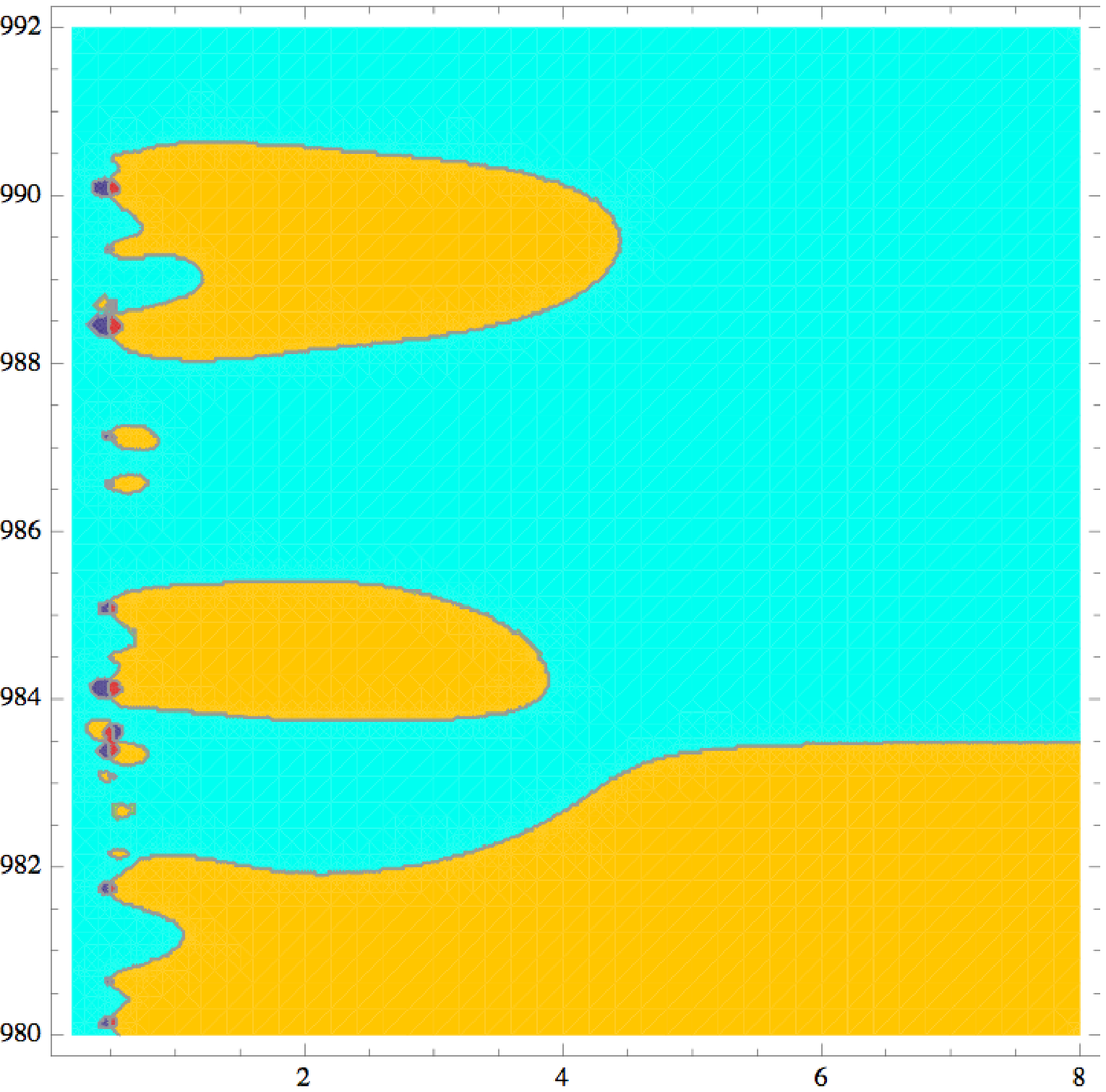}
\caption{ Contour plot of the phase of $\Delta_6(s)$, with its denominator having artificial off-axis zeros introduced as described in equation (\ref{app1}), and its numerator having a compensating number
of on-axis zeros introduced as described in equation (\ref{app2}). Here $s_0=0.45+983.5 i$,  $s_1=1/2+983.3 i$ and $s_2=1/2+983.7 i$ .}
\label{figapp2}
 \end{figure}

Thus, neither modification described provides an appropriate form for a legitimate counterexample. The failure of both is independent of the choice of the complex parameters. Indeed, the modification factor
in the second case has the expansion if $s\rightarrow \infty$ of
\begin{equation}
\frac{{\cal N}_6(s)}{{\cal D}_6(s)}=1-\frac{2 \delta\sigma_0^2+t_1^2+t_2^2-2 t_0^2}{s^2}+O(1/s^4).
\label{app3}
\end{equation}
Here $s_0=1/2+\delta\sigma_0+i t_0$, $s_1=1/2+i t_1$, $s_2=1/2+i t_2$. If the coefficient of the term in $1/s^2$ is required to be zero giving the required value of $t_0$, the expansion becomes
\begin{equation}
\frac{{\cal N}_6(s)}{{\cal D}_6(s)}=1-\frac{(4 \delta\sigma_0^2+(t_1-t_2)^2)(4 \delta\sigma_0^2+(t_1+t_2)^2)}{4 s^4}+O(1/s^6).
\label{app4}
\end{equation}
This algebraic term has a coefficient of the term in $1/s^4$ which can never be zero. Thus, the second modified form is  always incompatible with the exponential behaviour of $\Delta_6(s)$ as $s\rightarrow \infty$.

\end{document}